\documentclass{imsart}
\lastpage{1}

\usepackage{amsthm}
\usepackage{amsmath}
\usepackage{natbib}
\usepackage[colorlinks,citecolor=blue,urlcolor=blue,filecolor=blue,backref=page]{hyperref}
\usepackage{graphicx}
\usepackage{mathrsfs}
\usepackage{amsfonts}

\startlocaldefs

\numberwithin{equation}{section}

\theoremstyle{plain}
\newtheorem{thm}{Theorem}[section]
\theoremstyle{plain}
 \newtheorem{propo}{Proposition}[section]
 \theoremstyle{definition}
   \newtheorem{defi}{Definition}[section]
   \theoremstyle{definition} 
 \newtheorem{rmk}{Remark}[section]
\endlocaldefs

\begin{document}

\begin{frontmatter}
\title{Random Discrete Probability Measures Based on Negative Binomial Process}
\runtitle{Random Discrete Probability Measure Based on NBP}

\begin{aug}
\author{\fnms{Sadegh} \snm{Chegini}\thanksref{addr1}\ead[label=e1]{Sadegh.Chegini@uottawa.ca}}
\and
\author{\fnms{Mahmoud} \snm{Zarepour}\thanksref{addr1,t3}\ead[label=e2]{Mahmoud.Zarepour@uottawa.ca}}
%\and
%\author{\fnms{Third} \snm{Author}\thanksref{addr2,t1,m2}%
%\ead[label=e3]{third@somewhere.com}%
%\ead[label=u1,url]{http://www.foo.com}}

\runauthor{S. Chegini and M. Zarepour}

%\address[addr1]{Department of Mathematics and Statistics,
%University of Ottawa, 
%            Ottawa,
   %         Canada 
% \printead*{e1} % print email address of "e1"
  %  \printead*{e2}
% }  

\thankstext{addr1}{Department of Mathematics and Statistics,
University of Ottawa, 
            Ottawa,
            Canada 
 \printead*{e1} ~~% print email address of "e1"
    \printead*{e2}
 }

%\address[addr2]{Address of the Third author
%    Usually a few lines long
  %  Usually a few lines long
  %  \printead{e3}
  %  \printead{u1}
%}

%\thankstext{t1}{Some comment}
%\thankstext{t2}{First supporter of the project}
\thankstext{t3}{Corresponding author}

\end{aug}

\begin{abstract}
An important functional of Poisson random measure is the negative binomial process (NBP). We use NBP to introduce a generalized Poisson-Kingman distribution and its corresponding random discrete probability measure. This random discrete probability measure provides a new set of priors with more flexibility in nonparametric Bayesian models. It is shown how this random discrete probability measure relates to the nonparametric Bayesian priors such as Dirichlet process, normalized positive $\alpha$-stable process, Poisson-Dirichlet process (PDP), and others. An extension of the DP with its almost sure approximation is presented.
 Using our representation for NBP, we derive a new series representation for the PDP. 
\end{abstract}



\end{frontmatter}

\section {Introduction}
\label{section1}
%\bigskip

The Poisson random measure (or Poisson point process) relates to other important processes through its functionals. The setup for a general point process follows the exposition in \citet{Kellenberg1983} and \citet[Ch.3]{Resnick1987}.
Let $(\mathbb{E},\mathscr{E})$ be a locally compact space with a countable
basis with its associated Borel $\sigma$-algebra and also
let $(\mathbb{M},\mathscr{M})$ be the space of all point measures defined on $\mathbb{E}$
with its associated $\sigma$-algebra. 
A point process $\xi$ on $\mathbb{E}$ is a measurable map from the probability
space $(\Omega, \mathcal{F}, P) \to (\mathbb{M},\mathscr{M})$. It is well known that the probability law
of a process is uniquely determined by its Laplace functional. 
The Laplace functional of the point process $\xi$ is defined by
\begin{align*}
\Psi_{\xi}(f)=E(e^{-\xi(f)})=\int_{m\in \mathbb{M}}\mathrm{exp}\left\{-\int_{\mathbb{E}}f(x)m(dx) \right\} P(\xi \in \mathrm{d}m)
\end{align*}
where $f:\mathbb{E}\to [0,\infty)$ is measurable.

A point process $\xi$ on $(\mathbb{E}, \mathscr{E})$ is called a Poisson random measure with
mean measure  $\mu$, denoted by PRM$(\mu),$ if its random number of points in a set $A\in  \mathscr{E}$
has a Poisson distribution with parameter $\mu(A)$ and the numbers of points in disjoint
sets are independent random variables.
It can be shown that the Laplace functional of a PRM$(\mu)$ is given by
\begin{equation}
\Psi_{\xi}(f)=\mathrm{\exp}\left\{-\int_{\mathbb{E}}(1-e^{-f(x)})\mu(\mathrm{d}x)\right\}. \label{eq:2.0}
\end{equation}
\noindent
The following straightforward proposition derives a representation for the Poisson random measure with Lebesgue mean measure. 
There are many different ways to show this result. Since the recursive technique introduced in \citet{Banjevic2002} is helpful in other similar situations, we present it here.
\begin{propo}
Let $\xi\sim\mathrm{PRM}(\lambda)$ where $\lambda$ is the Lebesgue measure  on $[0,\infty)$. Then $\xi$ can be written as follows
\begin{equation}
\xi=\sum_{i=1}^{\infty}\delta_{\Gamma_{i}},\label{eq:2.1}
\end{equation}
where $$\Gamma_{i}=E_{1}+\cdots+E_{i},$$ and $(E_{i})_{i\geq1}$ is a sequence of independent and identically distributed (i.i.d.) random variables with an exponential distribution of mean 1. Throughout this paper, $\delta_X$ denotes the Dirac measure at $X$, i.e. $\delta_X(B)=1$ if $X\in B$ and $0$ otherwise.
\end{propo}
\begin{proof}
For any $t\geq 0$, define 
$$\xi_t=\sum_{i=1}^{\infty}\delta_{\Gamma_{i}+t}$$
such that $\xi_0=\xi$. Now, for any nonnegative function $f$,
\begin{align*}
\Psi_{\xi_t}(f)&=E(e^{-\xi_t(f)})=E(e^{-\sum_{i=1}^{\infty}f(\Gamma_{i}+t)})\\
&=E(E(e^{-\sum_{i=1}^{\infty}f(\Gamma_{i}+t)}|\Gamma_{1}=s))\\
&=\int_{0}^{\infty}e^{-f(s+t)}E(e^{-\sum_{i=1}^{\infty}f(\Gamma_{i}+s+t)})e^{-s}\mathrm{d}s\\
&=\int_{0}^{\infty}e^{-f(s+t)}\Psi_{\xi_{s+t}}(f)e^{-s}\mathrm{d}s.
\end{align*}
Using the change of variable $s+t = v$ and multiplying both sides by $e^{-t}$, we get
\begin{align*}
e^{-t}\Psi_{\xi_t}(f)=\int_{t}^{\infty}e^{-f(v)}\Psi_{\xi_v}(f)e^{-v}\mathrm{d}v.
\end{align*}
Differentiating both sides with respect to $t$, we get
\begin{align*}
-e^{-t}\Psi_{\xi_t}(f)+e^{-t}\frac{\partial\Psi_{\xi_t}(f)}{\partial t}&=-e^{-f(t)}\Psi_{\xi_t}(f)e^{-t}\\
\frac{\partial\Psi_{\xi_t}(f)}{\partial t}\frac{1}{\Psi_{\xi_t}(f)}&=1-e^{-f(t)}\\
\Psi_{\xi_t}(f)&=\mathrm{exp}\left(-\int_{t}^{\infty}(1-e^{-f(s)})\mathrm{d}s\right).
\end{align*}
Now, take $t = 0$ to get
\begin{align*}
\Psi_{\xi_0}(f)&=\mathrm{exp}\left(-\int_{0}^{\infty}(1-e^{-f(s)})\mathrm{d}s\right)
\end{align*}
which equals \eqref{eq:2.0} with $\mu(\mathrm{d}s)=\lambda (\mathrm{d}s)=\mathrm{d}s$. 
\end{proof}

\bigskip
\noindent
Applying Proposition 2.1 and 2.2 of \citet{Resnick1986} on $\mathrm{PRM}(\lambda)$ defined in \eqref{eq:2.1}, we can derive useful PRMs which in turn lead to other processes with applications in nonparametric Bayesian inference. 
First, take $T(x)=L^{-1}(x)$ where $L:(0,\infty)\to (0,\infty)$ is a decreasing bijection such that
$\sum_{i=1}^{\infty} L^{-1}(\Gamma_i)<\infty$, and 
$$ L^{-1}(y)=\inf \{x>0:L(x)\geq y\}. $$
Also, let $(\zeta_{i})_{i\geq 1}$ be a sequence of i.i.d. random elements in a Polish space $\mathbb{E}$ with a probability measure $H$ independent from $(\Gamma_{i})_{i\geq 1}$. Then we simply find that
\begin{align}
\sum_{i=1}^{\infty}\delta_{ L^{-1}(\Gamma_{i})}&\sim \mathrm{PRM}(L), \label{eq:2aaa}\\
\sum_{i=1}^{\infty}\delta_{(\zeta_{i}, L^{-1}(\Gamma_{i}))}&\sim \mathrm{PRM}(H\times L). \label{eq:2aaab}
\end{align}

\noindent
For example, $\sum_{i=1}^{\infty}\delta_{\Gamma_{i}^{-1/\alpha}}$
follows a PRM$(L)$ with 
\begin{equation}
L(x)=x^{-\alpha}=\int_{x}^{\infty}\alpha u^{-\alpha-1}\mathrm{d}u,~x>0,~ \alpha\in (0,1). \label{eq:2aacc}
\end{equation}
Throughout this paper, we use the function $L$  also as a measure with notation $L(\mathrm{d}x)=\mathrm{d} L(x)$. In fact, \eqref{eq:2aacc} denotes the L\'evy measure of the $\alpha$-stable random variable
$S_{\alpha}=\sum_{i=1}^{\infty}\Gamma_{i}^{-1/\alpha}.$
Notice that since $\Gamma_{i}/i\stackrel{a.s.}\longrightarrow 1$, $S_{\alpha}$ converges 
for $\alpha\in (0,1).$

\noindent
As another example, for $\theta>0$ and $x>0$, take 
\begin{equation}
L(x)=\theta\int_{x}^{\infty}u^{-1}e^{-u}\mathrm{d}u. \label{eq:2aab}
\end{equation}
Then a functional of  the random measure \eqref{eq:2aaab} given by
$Q=\sum_{i=1}^{\infty}L^{-1}(\Gamma_{i})\delta_{\zeta_{i}},$
is a gamma process denoted by GaP$(\theta, H)$. This means that for disjoint sets 
$A_{1}, \ldots, A_{k}$, the random variables
$\{Q(A_{i})\}_{1\leq i\leq k}$ are independent and $Q(A_{i})$ has a gamma distribution with shape parameter $\theta H(A_{i})$ and scale parameter of $1$.
Independence follows since $Q$ is a pure jump L\'evy process. See \citet{IshwaranZarepour2002} for more details.
This finite random measure is self-normalized as follows
\begin{equation}
P_{\theta,H}(\cdot)=\sum_{i=1}^{\infty}\frac{L^{-1}(\Gamma_{i})}{\sum_{i=1}^{\infty}
L^{-1}(\Gamma_{i})}\delta_{\zeta_{i}}(\cdot)  \label{eq:3.1}
\end{equation}
by \citet{Ferguson1973} to define the Dirichlet process DP$(\theta,H)$ to use it as a prior on the space of all probability measures on $\mathbb{E}$. The Dirichlet process is known as the cornerstone of the nonparametric Bayesian analysis. There has been an extensive effort to provide some generalizations and alternatives for this process; see for example, \citet{Pitman1997}, and \citet{Lijoi2005}. Also, see \citet{IshwaranZarepour2002} and \citet{Zarepour2012} for some alternative representations and approximations of this process.

Another important distribution which is resulted from a PRM is the Poisson-Kingman distribution. Consult \citet{Kingman1975} and \citet{Pitman2003} for properties and applications of this distribution. The vector of the normalized points of the PRM defined in equation \eqref{eq:2aaa}, which we call them Poisson-Dirichlet weights, will follow a Poisson-Kingman distribution denoted by PK$(L)$, i.e.
\begin{equation}
\left (\frac{L^{-1}(\Gamma_{1})}{\sum_{i=1}^{\infty}L^{-1}(\Gamma_{i})}, \frac{L^{-1}(\Gamma_{2})}{\sum_{i=1}^{\infty}L^{-1}(\Gamma_i)}, \ldots  \right ) \sim \mathrm{PK}(L) \label{eq:222bbb}
\end{equation}
defines a random discrete distribution on the infinite
dimensional simplex $\nabla_{\infty}:=\{(x_1,x_2,\ldots): x_i\geq 0,  i=1, 2, \ldots, \sum_{i=1}^{\infty}x_i=1\}$. As a particular case, if in \eqref{eq:222bbb} we take $L$ as the gamma L\'evy measure given in \eqref{eq:2aab}, then the Poisson-Dirichlet weights \eqref{eq:222bbb} are also said to have Poisson-Dirichlet distribution with parameter $\theta$, which we denote by $\mathrm{PD}(0,\theta)$. Also, as another special case, if in \eqref{eq:222bbb} one takes $L$ as the $\alpha$-stable L\'evy measure given in \eqref{eq:2aacc}, then the corresponding Poisson-Dirichlet weights \eqref{eq:222bbb} are also said to have Poisson-Dirichlet distribution with parameter $\alpha$, which we denote by $\mathrm{PD}(\alpha,0)$. Additionally, the random probability measure  \eqref{eq:3.1} is called a normalized $\alpha$-stable process if we employ the later Poisson-Dirichlet weights. See \citet{Ishwaran2001}.

Equivalently, a Poisson-Kingman distribution can be constructed using subordinators. Let $(X_t)_{t\geq 0}$ be a subordinator with L\'evy measure $L$ and write $(\Delta X_t:=X_t-X_{t-})_{t>0}$ for the jump process of $X_t$, and $\Delta X_t^{(1)}\geq\Delta X_t^{(2)}\geq \cdots$ for the ordered jumps up till time $t>0$. Then 
\begin{equation}
\left (\frac{\Delta X_t^{(1)}}{X_t}, \frac{\Delta X_t^{(2)}}{X_t}, \ldots  \right ) \sim \mathrm{PK}(tL). \nonumber
\end{equation}

We saw how using a PRM, the Poisson-Kingman distribution \eqref{eq:222bbb} is obtained and consequently how the random discrete probability measure \eqref{eq:3.1} can be constructed simply by using the Poisson-Dirichlet weights of the Poisson-Kingman distribution. In the rest of the paper, we will generalize the Poisson-Kingman distribution and the resulting random discrete probability measure by utilizing the negative binomial process instead of PRM. The negative binomial process representation which we use here is itself constructed directly from a PRM, unlike the representation in \citet{ross2018} where the negative binomial process is constructed from a trimmed subordinator.

The paper is organized as follows. In section \ref{section2}, we derive the negative binomial process as a functional of a PRM and then using this process, we generalize the random discrete probability measure \eqref{eq:3.1} and equivalently, the Poisson-Kingman distribution by adding a new parameter. As a special member of the family of the new defined random discrete probability measure, an extension of the Dirichlet process with its almost sure approximation is presented in section \ref{section4}.  Then, we present the general structures of the posterior and predictive processes in section \ref{section44}. A justification of the role of the parameter $r$ in clustering problem is given in section \ref{section45}.  In section \ref{section3}, we derive a new series representation for the Poisson-Dirichlet process \citep{Carlton1999}, which is based on our new representation of the negative binomial process.  In section \ref{section5}, we provide a simulation study to compare the efficiency of our suggested approximation of the new series representation of the Poisson-Dirichlet process with other representations of this process that exist in the literature. Finally, a summary of the conclusions is given in the last section.

\section {Negative Binomial Process}
\label{section2}

%\bigskip

In this section, we will see how the negative binomial process (NBP) is derived directly as a functional of a PRM. Later, we use this process to define a more general form of the Poisson-Kingman distribution and its corresponding random discrete probability measure. First, we note that for any constant $c>0$, a simple use of Proposition 2.1 of \citet{Resnick1986} shows that the process
$\sum_{i=1}^{\infty}\delta_{\Gamma_{i}+c}$
is a PRM$(\lambda)$ on $\mathbb{E}=[c,\infty)$, and $\sum_{i=1}^{\infty}\delta_{\Gamma_{i}/c}$
follows PRM$(c\lambda)$ on $\mathbb{E}=[0,\infty)$. Now, for any nonnegative integer $r$ and setting $\Gamma_0=1,$ consider the random measure
\begin{equation}
\eta=\sum_{i=r+1}^{\infty}\delta_{\Gamma_i/\Gamma_r}. \nonumber
\end{equation}
First, note that conditional on $\{\Gamma_r=u\}$, the process $\eta$
follows PRM$(u\lambda)$ on $\mathbb{E}=(1,\infty)$. So, the Laplace functional of $\eta$ is

$$E(e^{-\eta(f)})=E\left [E(e^{-\eta(f)}|\Gamma_r=u)\right ]$$
$$=\int_{0}^{\infty}E(e^{-\eta(f)}|\Gamma_r=u) P(\Gamma_r\in \mathrm{d}u)$$
$$=\int_{0}^{\infty}\exp\left\{-\int_{1}^{\infty}(1-e^{-f(x)})u\lambda(\mathrm{d}x)\right\} P(\Gamma_r\in \mathrm{d}u)$$
$$=\int_{0}^{\infty}\exp\left\{-u\int_{1}^{\infty}(1-e^{-f(x)})\lambda(\mathrm{d}x)\right\} \frac{u^{r-1}e^{-u}}{\Gamma(r)} \mathrm{d}u$$
$$=\left(1+\int_{1}^{\infty}(1-e^{-f(x)})\lambda(\mathrm{d}x)\right)^{-r}.$$
This is in fact the Laplace functional of the negative binomial process defined in  
\citet{Gregoire1984}. We denote this process by NBP$(r, \lambda)$ and write $\eta\sim\mathrm{NBP}(r, \lambda)$ on $\mathbb{E}=(1,\infty)$. Following up on this example, we state the following theorem.

\bigskip
\noindent
\begin{thm}
With a decreasing bijection $L:(0,\infty)\rightarrow (0,\infty)$ such that
$\sum_{i=1}^{\infty} L^{-1}(\Gamma_i)<\infty$, the following point process
\begin{equation}
\kappa=\sum_{i=r+1}^{\infty}\delta_{L^{-1}(\Gamma_i/\Gamma_r)} \label{eq:3aa}
\end{equation}
follows an $\mathrm{NBP}(r, L)$ on $\mathbb{E}=(0,L^{-1}(1))$. 
\end{thm}

\noindent
The proof of the theorem is similar to what was presented above for the process $\eta$.
Three important examples of $L$ are positive $\alpha$-stable, gamma, and inverse-Gaussian L\'evy measures \citep{Lijoi2005,Labadi2013,Labadi2014}.

\noindent
The NBP was defined in \citet{Gregoire1984} only through its Laplace functional and no point process or subordinator representation was provided.  
As it is seen, the point process representation of NBP$(r, L)$ in \eqref{eq:3aa} was derived directly as a functional of a PRM. In \citet{ross2018}, a point process representation of the NBP, which equals to \eqref{eq:3aa} in distribution, was derived using ordered jumps of a trimmed subordinator. If $(X_t)_{t\geq 0}$ is a subordinator with L\'evy measure $L$, by writing $(\Delta X_t:=X_t-X_{t-})_{t> 0}$ for the jump process of $X_t$, and $\Delta X_t^{(1)}\geq\Delta X_t^{(2)}\geq \cdots$ for the ordered jumps at $t>0$, then the point process $\mathbb{B}^{(r)}=\sum_{i=1}^{\infty}\delta_{J_r(i)}$ follows NBP$(r,L)$ where $J_r(i)=\frac{\Delta X_1^{(r+i)}}{\Delta X_1^{(r)}},i=1,2,\ldots$.

\begin{rmk}
In the literature, the terminology “negative binomial process” is used for mathematically distinct concepts. Therefore, it seems necessary to clarify these concepts to avoid confusion. As stated before, in this paper, \citet{ross2018}, and \citet{ross2020,ross2021,Ipsen2018}, the negative binomial process is the one defined in \citet{Gregoire1984}.
However, in engineering and computer science literature, some other definitions of the negative binomial process can be found. For example, the  negative binomial process defined in \citet{Zhou2015} is different from the one defined in  \citet{Zhou2012} and \citet{Broderick2015}, which all are different from Gregoire's definition that we are considering in this paper.
\end{rmk}

Following Definition 2.1 of \citet{ross2018}, by normalizing the points of the NBP defined in \eqref{eq:3aa}, the following sequence
\begin{equation}
\left (\frac{L^{-1}(\Gamma_{r+1}/\Gamma_r)}{\sum_{i=r+1}^{\infty}L^{-1}(\Gamma_{i}/\Gamma_r)}, \frac{L^{-1}(\Gamma_{r+2}/\Gamma_r)}{\sum_{i=r+1}^{\infty}L^{-1}(\Gamma_i/\Gamma_r)}, \ldots  \right )  \label{eq:4}
\end{equation}
defines a 2-parameter random discrete distribution on the infinite
dimensional simplex $\nabla_{\infty}$. In \citet{ross2018}, this distribution is called a Poisson-Kingman distribution generated by NBP$(r,L)$ and is denoted by $\mathrm{PK}^{(r)}(L)$. In particular case of $\alpha$-stable L\'evy measure, this process is denoted by $\mathrm{PD}_{\alpha}^{(r)}$. Clearly, it is seen that the random sequence \eqref{eq:4} equals 
\begin{equation}
\left (\frac{J_r(1)}{\sum_{i=1}^{\infty}J_r(i)}, \frac{J_r(2)}{\sum_{i=1}^{\infty}J_r(i)}, \ldots  \right )  \nonumber
\end{equation}
in distribution. Also, as pointed out in \citet{ross2018}, an NBP$(r,L)$ can be characterized as a PRM with randomized intensity measure $\Gamma_rL$ where $\Gamma_r$ is a Gamma$(r,1)$ random variable, i.e. $\mathrm{PRM}(\Gamma_rL)\stackrel{d}= \mathrm{NBP}(r,L)$. 
In other words, a gamma subordinated L\'evy process $X_{\sigma_r}$ will follow an NBP$(r,L)$ where $(\sigma_r)_{r>0}$ is an independent gamma subordinator having L\'evy measure \eqref{eq:2aab} with $\theta=1$.   
Then by the definition of the Poisson-Kingman distribution, the random sequence \eqref{eq:4} will also equal in distribution to  
\begin{equation}
\left (\frac{\Delta X_{\sigma_r}^{(1)}}{X_{\sigma_r}}, \frac{\Delta X_{\sigma_r}^{(2)}}{X_{\sigma_r}}, \ldots  \right ).  \label{eq:4444}
\end{equation}
 
\begin{defi}
Let $(\zeta_i)_{i\geq 1}$ be i.i.d. random variables with values in $\mathbb{E}$ and common distribution $H$, then we may introduce the following random discrete probability measure on $E$ as a functional of \eqref{eq:3aa} or using the sequence  \eqref{eq:4} as follows
\begin{equation}
P_{r,L,H}(\cdot)=\sum_{i=r+1}^{\infty}\frac{L^{-1}(\Gamma_{i}/\Gamma_{r})}{\sum_{i=r+1}^{\infty}
L^{-1}(\Gamma_{i}/\Gamma_{r})}\delta_{\zeta_{i}}(\cdot).  \label{eq:3.4}
\end{equation}
We employ the notation $ \mathrm{PKP}^{(r)}(H;L)$ for the distribution of the random discrete probability measure defined in  \eqref{eq:3.4} and we write $P_{r,L,H}\sim \mathrm{PKP}^{(r)}(H;L)$. 
\end{defi}

 \section {Extended Dirichlet Process and its Approximation}
\label{section4}

In the case $L$ is the $\alpha$-stable L\'evy measure given in \eqref{eq:2aacc}, $\mathrm{PK}^{(r)}(L)$  has been investigated thoroughly. For example, it is shown that how this distribution relates to other Poisson–Dirichlet models by letting $r \to \infty$ in \citet{ross2020}. Also, in \citet{Ipsen2018}, this distribution is fitted to gene and species sampling data, demonstrating the utility of allowing the extra parameter $r$ in data analysis. 

We  may now take the probability measure \eqref{eq:3.4} with $L$ as the gamma L\'evy measure defined in \eqref{eq:2aab} to develop a prior distribution on the space of all probability distributions. This prior would be a natural extension of the Dirichlet process (the Dirichlet process is recovered when $r=0$). In the following theorem, we provide an efficient approximation for our extended Dirichlet process.

\bigskip
\noindent
\begin{thm}
Let $W_n$ be a random variable with distribution Gamma$(\alpha/n, 1)$. Define

\[
G_n(x)=\mathrm{Pr}(W_n>x)=\int_x^{\infty}\frac{1}{\Gamma(\alpha/n)}e^{-t}t^{\alpha/n-1}\mathrm{d}t
\]
and
\[ G_n^{-1}(y)=\mathrm{inf}\{x: G_n(x)\geq y\},~~0<y<1. \]
Let $L$ be the gamma L\'evy measure \eqref{eq:2aab} and $(\zeta_i)_{i\geq1}$ be a sequence of i.i.d. random variables with values in $\mathbb{E}$ and common distribution $H$, independent of
$(\Gamma_i)_{i\geq1}$, then as $n\to \infty$
\[
P_{n,r,H}=\sum_{i=r+1}^{n}\frac{G_n^{-1}\left (\frac{\Gamma_{i}}{\Gamma_{r}\Gamma_{n+1}}\right )}
{\sum_{i=r+1}^{n}G_n^{-1}\left (\frac{\Gamma_{i}}{\Gamma_{r}\Gamma_{n+1}}\right )}\delta_{\zeta_{i}}     \stackrel{a.s.}\longrightarrow       P_{r,L,H}=\sum_{i=r+1}^{\infty}\frac{L^{-1}\left (\frac{ \Gamma_{i}}{\Gamma_{r}}\right )}{\sum_{i=r+1}^{\infty}
L^{-1}\left (\frac{\Gamma_{i}}{\Gamma_{r}}\right )}\delta_{\zeta_{i}} 
\]
on $\mathbb{E}$ with respect to the weak topology.
\end{thm}
\noindent
\begin{proof}
The proof is similar to that of Theorem 1 in \citet{Zarepour2012} and the fact that $G_n^{-1}\left (\frac{x}{cn}\right ) \stackrel{a.s.}\longrightarrow L^{-1}\left (\frac{x}{c}\right )$. Consequently $G_n^{-1}\left (\frac{\Gamma_{i}}{\Gamma_{r}\Gamma_{n+1}}\right ) \stackrel{a.s.}\longrightarrow L^{-1}\left (\frac{ \Gamma_{i}}{\Gamma_{r}}\right )$ by taking the constant $c=\Gamma_{r}$, $x=\Gamma_{i}$, and $n=\Gamma_{n+1}$ as $\Gamma_{n+1}/n \stackrel{a.s.}\longrightarrow 1$ when $n\to \infty$. 
\end{proof}
\noindent
Our proposed approximation has several advantages. For example, our representation avoids the use of an infinite sum and instead our finite many weights are simply the quantile functions of the Gamma$(\alpha/n, 1)$ distribution evaluated at $1-\Gamma_{i}/\Gamma_{r}\Gamma_{n+1}$. In previous representation, it is necessary to calculate  $L^{-1}$ which can not be written in a closed form. In addition, our introduced weights are stochastically decreasing contrary to stick-breaking weights in  \citet{ross2018}. A similar proposal for the Dirichlet process can be found in \citet{Zarepour2012}.

\section {Posterior and Predictive Distribution}
\label{section44}
To develop a full Bayesian analysis, we can generalize $P_{r,L,H}$ defined in \eqref{eq:3.4} by assuming that $r$ is a realization of a random variable $R$, on set of non-negative integers with an arbitrary probability mass function $\pi(r).$
Also for simplicity in notations, denote the weights of $P_{R,L,H}$ by $p_i$. Therefore, we can write
\[ P_{R,L,H}=\sum_{i=R+1}^{\infty}p_i\delta_{\zeta_i}. \]
For given observations from $P_{R,L,H}$, the posterior and predictive distribution of the prior $P_{R,L,H}$ can be obtained from \citet{Ongaro2004} using a recursive method. This method obtains the posterior distribution for a general random discrete probability measure of the form 
\[ P=\sum_{i=1}^{M}p_i\delta_{\zeta_i}, \]
where $M$ is an extended integer valued random variable with an arbitrary probability distribution $p(m).$ Moreover, conditionally on $M,~(p_1,\ldots,p_M)$ has an arbitrary distribution $Q_{M}$ on simplex $\nabla_{M}:=\{(x_1,\ldots,x_M): x_i\geq 0,  i=1, 2, \ldots,M, \sum_{i=1}^{M}x_i=1\}$. The random positions $\zeta_i$'s are i.i.d. from a diffuse probability measure $H$ and are independent of all other random elements. 

In our case, for $P_{R,L,H}$, we only need to change the role of $M$ with $R$, where similar results of  \citet{Ongaro2004} follow easily.
Following their procedure, we need to find the posterior distribution of random elements of $P_{R,L,H}$, i.e. $(R,\boldsymbol{p},\boldsymbol{\zeta})$ where $\boldsymbol{p}=(p_{R+1},p_{R+2},\ldots)$ and $\boldsymbol{\zeta}=(\zeta_{R+1},\zeta_{R+2},\ldots)$.
The calculations for finding the posterior of $(R,\boldsymbol{p},\boldsymbol{\zeta})$ remain similar to the ones in Propositions 4 and 5 in \citet{Ongaro2004}. The final summary is provided in the following theorem.

\noindent
\begin{thm}
Let $\boldsymbol{X} = (X_1,\ldots,X_n)$ be a random sample of $n$ observations from $P_{R,L,H}$. The posterior process $P_{R,L,H}|\boldsymbol{X} $ can be represented as
\begin{equation}
\left (P_{R,L,H}|\boldsymbol{X} \right )=\sum_{i=1}^{k}\gamma_i^{\boldsymbol{X}}\delta_{X_i^{\ast}}+\sum_{i=R^{\boldsymbol{X}}+1}^{\infty}p_i^{\boldsymbol{X}}\delta_{\zeta_i}, \label{eq:5.1}
\end{equation}
where $X_i^{\ast}$s are the distinct values among the observations $\boldsymbol{X}$ and $R^{\boldsymbol{X}}$ and \\
$\boldsymbol{p}^{\boldsymbol{X}}=(\gamma_1^{\boldsymbol{X}},\ldots,\gamma_k^{\boldsymbol{X}}, p_{R^{\boldsymbol{X}}+1}^{\boldsymbol{X}},p_{R^{\boldsymbol{X}}+2}^{\boldsymbol{X}},\ldots)$ denote the posteriors of $R$ and $\boldsymbol{p}$, respectively. The distribution of $R^{\boldsymbol{X}}$ and $\boldsymbol{p}^{\boldsymbol{X}}$ are obtained using a recursive method similar to Corollaries 3 and 4 in \citet{Ongaro2004}. 
\end{thm}
\noindent
To calculate the predictive distribution, take expectation of \eqref{eq:5.1} to get
\[\mathrm{Pr}\left \{X_{n+1}\in A|\boldsymbol{X} \right \}=\sum_{i=1}^{k}c_i^{\boldsymbol{X}}\delta_{X_i^{\ast}}(A)+(1-c_1^{\boldsymbol{X}}-\cdots -c_k^{\boldsymbol{X}})H(A), \]
where $c_i^{\boldsymbol{X}}=E(\gamma_i^{\boldsymbol{X}})$ for $i=1,\ldots,k$.

\section { Applications}
\label{section45}

Sufficiently large sample from a random discrete probability measure like \eqref{eq:3.1} or \eqref{eq:3.4}, always includes ties with positive probability. Let $K_n=k\in \{1,\ldots,n\}$ be the number of distinct values among $n$ observations.  Denote $X_1^{\ast},\ldots,X_k^{\ast}$ as distinct values among observations $X_1,\dots, X_n$. Moreover, take $n_j=\sum_{i=1}^{n}I(X_i=X_j^{\ast})$ for $j=1,2,\ldots,k$. Obviously, $\sum_{j=1}^k n_j =n$.  
Note that we already saw how $k$ appeared in the posterior process \eqref{eq:5.1}. For the Dirichlet process, $k$ grows slowly as $n\to \infty.$ 
 As it is shown in \citet{Korwar1973} and \citet{Pitman2006}, if 
$$(X_1,\ldots,X_n)|P\sim P,~~P\sim\mathrm{DP}(\theta,H)$$
then
$$K_n/\log(n)\stackrel{a.s.}\longrightarrow \theta~~ \mathrm{as}~~ n \to \infty.$$ 
It means that the random number of distinct values $ K_n$ grows only in a logarithmic fashion.
In other perspective, the Dirichlet process prior assigns most of the largest weights to its initial points.
 This property causes inflexibility in the use of the Dirichlet process as a prior mixing measure in nonparametric Bayesian hierarchical mixture models (or so called density estimation problem) when it is fitted to an over-dispersed data \citep{Lo1984,Escobar1995, Ishwaran2001,Lijoi2005,Lijoi2007}.  Adding the new parameter $r$ and working with \eqref{eq:3.4} instead of \eqref{eq:3.1}, will allow to remove those initial large probability weights and produce smother ones instead. For the Dirichlet case, Table \ref{table2} shows how choosing larger values for $r$ will lead to smoother probability weights in \eqref{eq:3.4}.  
Specially, in problems of density estimation similar to \cite{Lo1984}, this flexibility plays a crucial role but we do not address them here in this paper.
\begin{table}[t!]
\caption{The first ten probability weights in \eqref{eq:3.4} when $L$ is the gamma L\'evy measure \eqref{eq:2aab} with $\theta=3$.}
\label{table2}\par
\begin{tabular}{cccc}
\hline

$r=0$ & \multicolumn{1}{c}{$r=3$} & \multicolumn{1}{c}{$r=5$} & \multicolumn{1}{c}{$r=10$}  \\
\hline
0.367597022 & 0.24369485 & 0.16427045 & 0.06353087   \\
0.168573239 & 0.23947841 & 0.14319002 & 0.05886303   \\
0.165457071 & 0.10281577 & 0.12842541 & 0.05418117 \\
0.149080111 & 0.08716432 & 0.10524699 & 0.05112053  \\
0.058821776 & 0.07639968 & 0.07391248 & 0.04858739   \\
0.056183134 & 0.05990201 & 0.06339345 & 0.04626107 \\
0.012551887 & 0.03862790 & 0.05279059 & 0.04349017  \\
0.007812625 & 0.03184211 & 0.04298598 & 0.03795044  \\
0.003792634 & 0.02524151 & 0.03705701 & 0.03606621 \\
0.001971704 & 0.01939928 & 0.03245281 & 0.03021665

\\ \hline
\end{tabular}
\end{table}

If we use $\alpha$-stable L\'evy measure \eqref{eq:2aacc} in \eqref{eq:3.1}, this property is even more obvious. Since for the $\alpha$-stable random variable $S_{\alpha}=\sum_{i=1}^{\infty}\Gamma_{i}^{-1/\alpha},~\alpha \in (0,1),$ taking for example $\alpha=0.5$, the first four terms of $S_{1/2}$ have infinite variance which means that there is a huge fluctuation among the initial terms.

To present an alternative interpretation,  notice that, the negative binomial distribution is preferred to the Poisson distribution when the data are over-dispersed. The variance of the negative binomial distribution is larger than its mean while for the Poisson distribution, both mean and variance are equal. Therefore, we expect that the random discrete probability measure \eqref{eq:3.4} outperforms \eqref{eq:3.1} as a prior mixing measure in nonparametric Bayesian hierarchical mixture models fitted to the over-dispersed data. Recall that the weights of  \eqref{eq:3.4} are the normalized points of the negative binomial process \eqref{eq:3aa}. However, the weights of \eqref{eq:3.1} are the normalized points of the Poisson process \eqref{eq:2aaa} with a certain mean measure. We can observe that the weights in \eqref{eq:3.4} decrease much slower than that of \eqref{eq:3.1}. See Table \ref{table2}  for the case that $L$ is the gamma L\'evy measure \eqref{eq:2aab}. This exhibits the mechanism that the random discrete probability measure \eqref{eq:3.4} can naturally capture the over-dispersion better. In \citet[Theorem 2.1]{ross2021}, the growth rate of  $K_n(\alpha,r)=K_n$ is given rigorously for $\mathrm{PK}^{(r)}(L)$ when $L$ is the $\alpha$-stable L\'evy measure. In other words, as $n\to \infty$
 \begin{equation}
 K_n/n^{\alpha}\stackrel{d}\longrightarrow Y_{\alpha,r}. \label{eq:mmm} 
 \end{equation}
See equation 2.8 in  \citet{ross2021} for the distribution of $Y_{\alpha,r}$. The growth rate $n^{\alpha}$ in \eqref{eq:mmm} is equal to that of the Poisson-Dirichlet process \citep[Theorem 3.8]{Pitman2006} and the normalized generalized gamma process \citep[Proposition 3]{Lijoi2007}.
It is not surprising to see that the growth rates of $K_n$ for these processes are equal as all these processes belong to the greater family of the random discrete probability measure \eqref{eq:3.4}.
 We notice that the Poisson-Dirichlet process $\mathrm{PDP}(H;\alpha,\theta)$  is a particular case of the random discrete probability measure \eqref{eq:3.4} with $r=\theta/\alpha$ and $L$ given in \eqref{eq:3.4aaaa} and the normalized generalized gamma process is a particular case of the random discrete probability measure \eqref{eq:3.4} with $r=0$ and $L$ given in \eqref{eq:3.4aaaa}. 
 The Poisson-Dirichlet process and the normalized generalized gamma process have already been recommended in the literature to be exploited as mixing measures in nonparametric Bayesian hierarchical mixture models in order to allow the number of distinct values (the number of clusters) increases at a rate faster than that of the Dirichlet process. 
  Therefore, the proposed random discrete probability measure \eqref{eq:3.4} can be considered as a general alternative for the mixing measure in nonparametric Bayesian hierarchical mixture models fitted to the over-dispersed data.

\section {A New Alternative Series Representation for the Poisson-Dirichlet Process}
\label{section3}

Using NBP,  we can find another series representation rather than the stick-breaking representation for the Poisson-Dirichlet process. See \citet{Carlton1999} for properties and applications of this process in nonparametric Bayesian analysis. 
For $0\leq \alpha<1, ~\theta >-\alpha$, let  $(\beta_k)_{k\geq 1}$ be a sequence of independent random variables, where $\beta_k$ has the Beta$(1-\alpha,\theta+k\alpha)$ distribution. If we define
$$p'_1=\beta_1,~p'_i=\beta_i\prod_{k=1}^{i-1}(1-\beta_k),~i\geq 2,$$
then the ranked sequence of  $(p'_i)_{i\geq 1}$ denoted by $p_1\geq p_2\geq\ldots$  is said to have a Poisson-Dirichlet distribution with parameters $\alpha$ and $\theta$ denoted by PD$(\alpha, \theta)$ and we write $(p_1,p_2,\ldots)\sim \mathrm{PD}(\alpha, \theta).$ Also, note that $(p'_1,p'_2,\ldots)\sim \mathrm{GEM}(\alpha, \theta)$ with the notation used in \citet{Carlton1999}.
 Moreover, let  $(\zeta_i)_{i\geq 1}$ be i.i.d. random variables with values in $\mathbb{E}$ and common distribution $H$. 
 Then the random probability measure 
\begin{equation}
P_{\alpha,\theta,H}(\cdot)=\sum_{i=1}^{\infty}p_i\delta_{\zeta_i}   \label{eq:3b4}
\end{equation}
is called Poisson-Dirichlet process
with parameters $\alpha, \theta,$ $H,$ and denoted by $\mathrm{PDP}(H;\alpha,\theta)$.
As it is shown in \citet[Lemma 2.1]{Labadi2014}, the $p'_i$'s are not strictly decreasing almost surely. Therefore, the stick-breaking representation \eqref{eq:3b4} is inefficient for simulating this process due to failure in proper stopping rules. Another approach based on Proposition 22 of  \citet{Pitman1997}, is proposed in \citet{Labadi2014} for simulating this process. This approach is more accurate than the method based on the stick-breaking representation, however, it includes more complex steps in its algorithm. In this section, we apply our new representation of NBP \eqref{eq:3aa} to the Proposition 21 in \citet{Pitman1997} to give a new representation for the Poisson-Dirichlet process. In section \ref{section5}, we will show how simulating the Poisson-Dirichlet process using this representation is much more efficient while it avoids the shortcoming of the stick-breaking representation for simulation purposes. Moreover, our approach is less complex compared with the algorithm A in \citet{Labadi2014}.

Now, following the Proposition 21 in \citet{Pitman1997}, for $\theta >0$ and $0<\alpha<1$, let $(X_t)_{t\geq 0}$ be a subordinator having L\'evy measure
\begin{equation}
L(x)=\frac{\alpha}{\Gamma(1-\alpha)}\int_{x}^{\infty}u^{-\alpha-1}e^{-u}\mathrm{d}u,~x>0 \label{eq:3.4aaaa}
\end{equation}
and let $(\sigma_r)_{r> 0}$ be an independent gamma subordinator. Then 
\begin{equation}
\left (\frac{\Delta X_{T}^{(1)}}{X_{T}}, \frac{\Delta X_{T}^{(2)}}  {X_{T}}, \ldots  \right )\sim \mathrm{PD}(\alpha, \theta)~~\mathrm{if}~~ T=\sigma_{\theta/\alpha}. \label{eq:3.444b}
\end{equation}
Comparing  \eqref{eq:3.444b} with  \eqref{eq:4444}, we see that $\mathrm{PD}(\alpha, \theta)\stackrel{d}=    \mathrm{PK}^{(\theta/\alpha)}(L)$ with $L$ given in \eqref{eq:3.4aaaa}. 
Since $\mathrm{PD}(\alpha, \theta)\stackrel{d}=    \mathrm{PK}^{(\theta/\alpha)}(L)$ and  \eqref{eq:4444} is equal to \eqref{eq:4} in distribution, we can conclude  
$\mathrm{PDP}(H;\alpha,\theta)\stackrel{d}=    \mathrm{PKP}^{(\theta/\alpha)}(H;L)$ for $\theta >0, 0<\alpha<1$ and $L$ in \eqref{eq:3.4aaaa}. In other words, the random probability measure
\begin{equation}
P_{r,L,H}(\cdot)=\sum_{i=r+1}^{\infty}\frac{L^{-1}(\Gamma_{i}/\Gamma_{r})}{\sum_{i=r+1}^{\infty}
L^{-1}(\Gamma_{i}/\Gamma_{r})}\delta_{\zeta_{i}}(\cdot)      \label{eq:3.55555}
\end{equation}
with $L$ given in  \eqref{eq:3.4aaaa} and $r=\theta/\alpha$ is distributed as either $\mathrm{PDP}(H;\alpha,\theta)$ or $\mathrm{PKP}^{(\theta/\alpha)}(H;L)$.
 Therefore, \eqref{eq:3.55555} provides another series representation for the Poisson-Dirichlet process for the case $\theta >0$ and $0<\alpha<1$ through a negative binomial process.
 
Note that $\mathrm{PK}^{(0)}(L)\stackrel{d}= \mathrm{PD}(\alpha, 0)$ for $0<\alpha<1$ and $L$ given in \eqref{eq:2aacc} and also, $   \mathrm{PK}^{(0)}(L)\stackrel{d}= \mathrm{PD}(0, \theta)$ for $\theta >0$ and $L$ given in \eqref{eq:2aab}. See \citet{ross2018} for stick-breaking representations of $\mathrm{PK}^{(r)}(L)$ with $L$ given in \eqref{eq:2aacc} and  \eqref{eq:2aab} when $r>0$.

\section {Simulating a New Approximation of the Poisson-Dirichlet Process}
\label{section5}
By applying a truncation method on the new series representation of the Poisson-Dirichlet process $\mathrm{PDP}(H;\alpha,\theta)$ given in \eqref{eq:3.55555}, we can approximate this process  from
 \begin{equation}
P_{n,r,L,H}(\cdot)=\sum_{i=r+1}^{n}\frac{L^{-1}(\Gamma_{i}/\Gamma_{r})}{\sum_{i=r+1}^{n}
L^{-1}(\Gamma_{i}/\Gamma_{r})}\delta_{\zeta_{i}}(\cdot)  \label{eq:3.4cccc}
\end{equation}
 for $0<\alpha <1$, $\theta > 0$, $r=\theta/\alpha$, and $L$ given in \eqref{eq:3.4aaaa}. We can suggest a stopping rule for choosing $n=n(\epsilon)$ as follows           
 $$n=\inf \left \{i:\frac{L^{-1}(\Gamma_{i}/\Gamma_{r})}{ \sum_{j=r+1}^{i}L^{-1}(\Gamma_{j}/\Gamma_{r})}<\epsilon \right \} ~~\mathrm{for}~\epsilon \in (0,1).$$
 
   % figure 1
\begin{figure} [t!]
\includegraphics[scale=0.7]{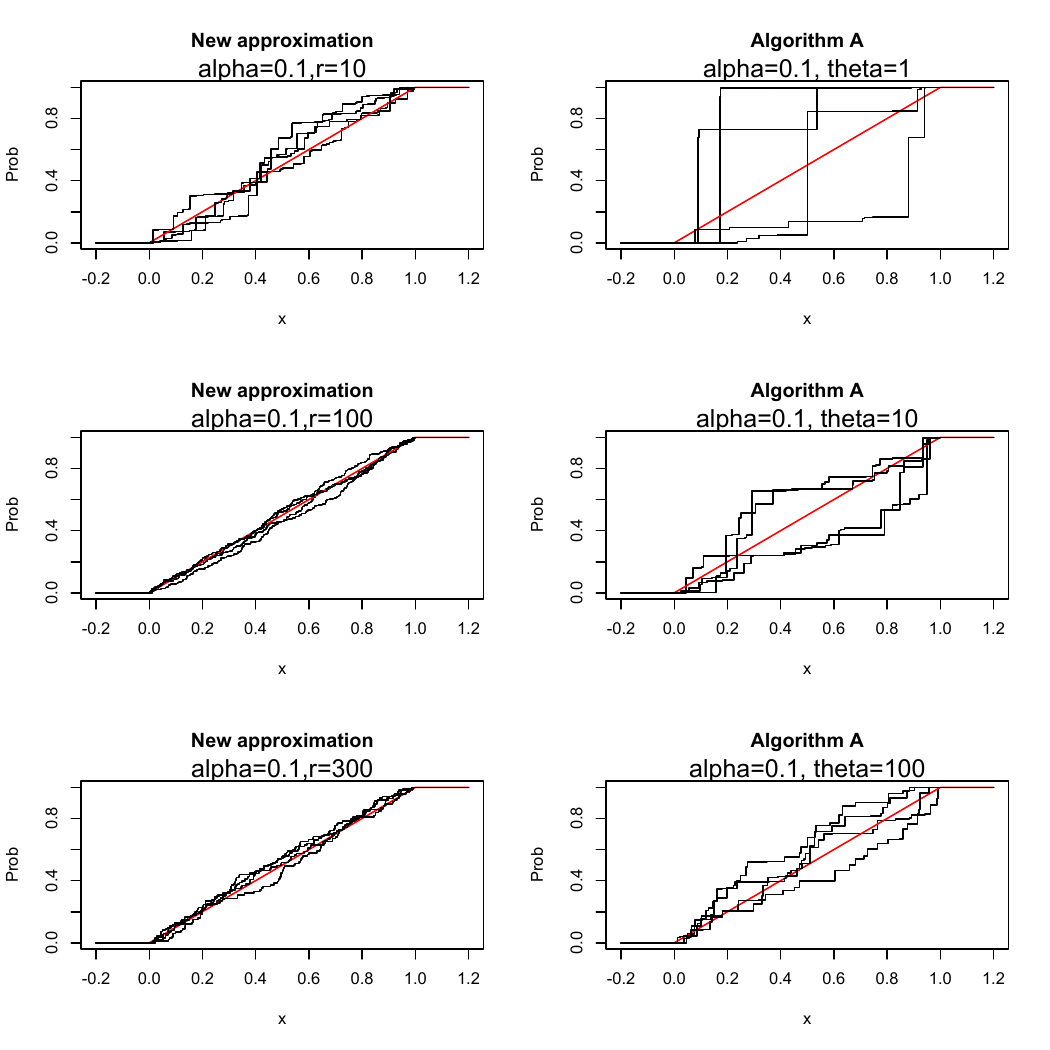}
\caption[]{Sample paths of the two-parameter Poisson-Dirichlet process, where $H$ is the uniform distribution on $[0, 1]$, $\alpha= 0.1,$ and $\theta = 1, 10, 100$. The red line denotes the cumulative distribution function of $H$.}
\label{fig1}
\end{figure}

% figure 2
\begin{figure}[t!] 
\includegraphics[scale=0.7]{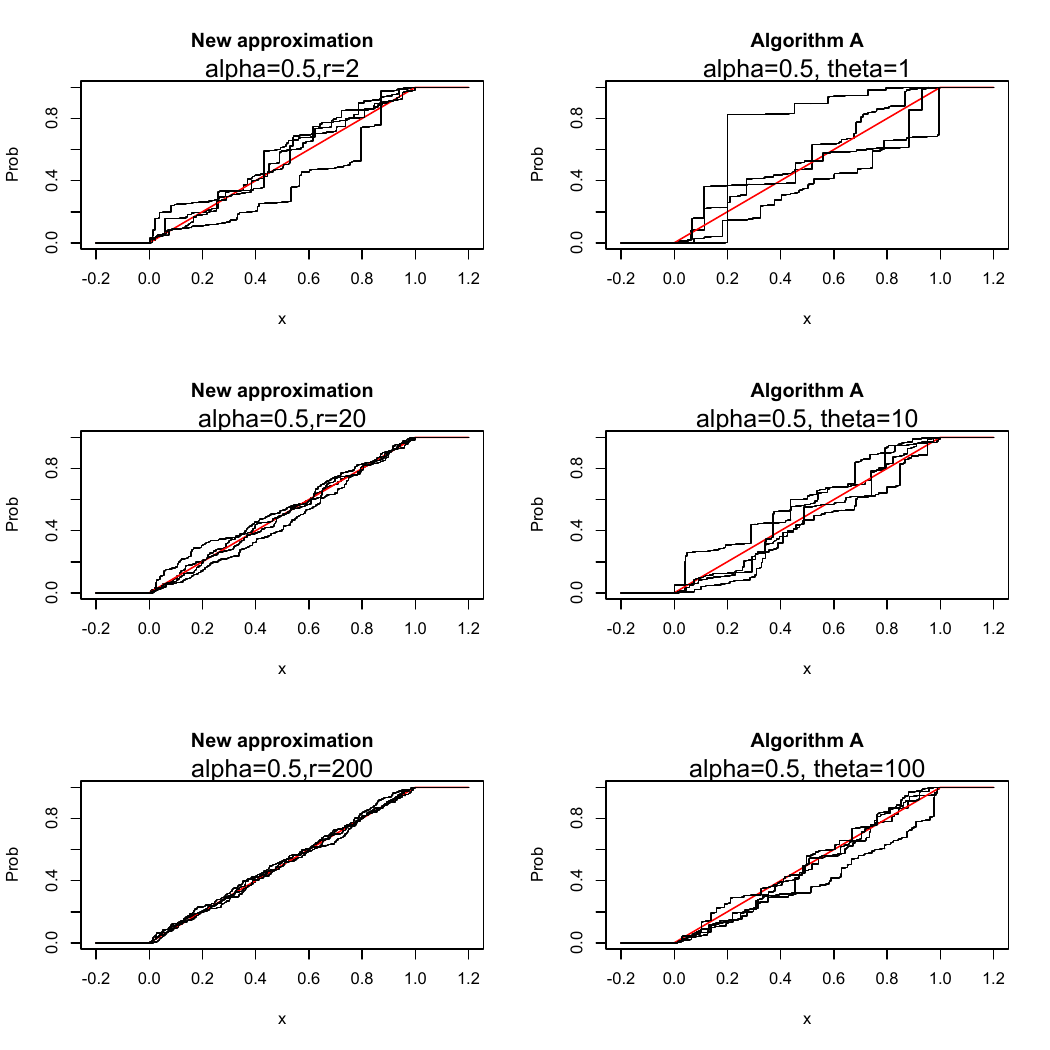}
\caption[]{Sample paths of the two-parameter Poisson-Dirichlet process, where $H$ is the uniform distribution on $[0, 1]$, $\alpha= 0.5,$ and $\theta = 1, 10, 100$. The red line denotes the cumulative distribution function of $H$.}
\label{fig2}
\end{figure}

% figure 3
\begin{figure} [t!]
\includegraphics[scale=0.7]{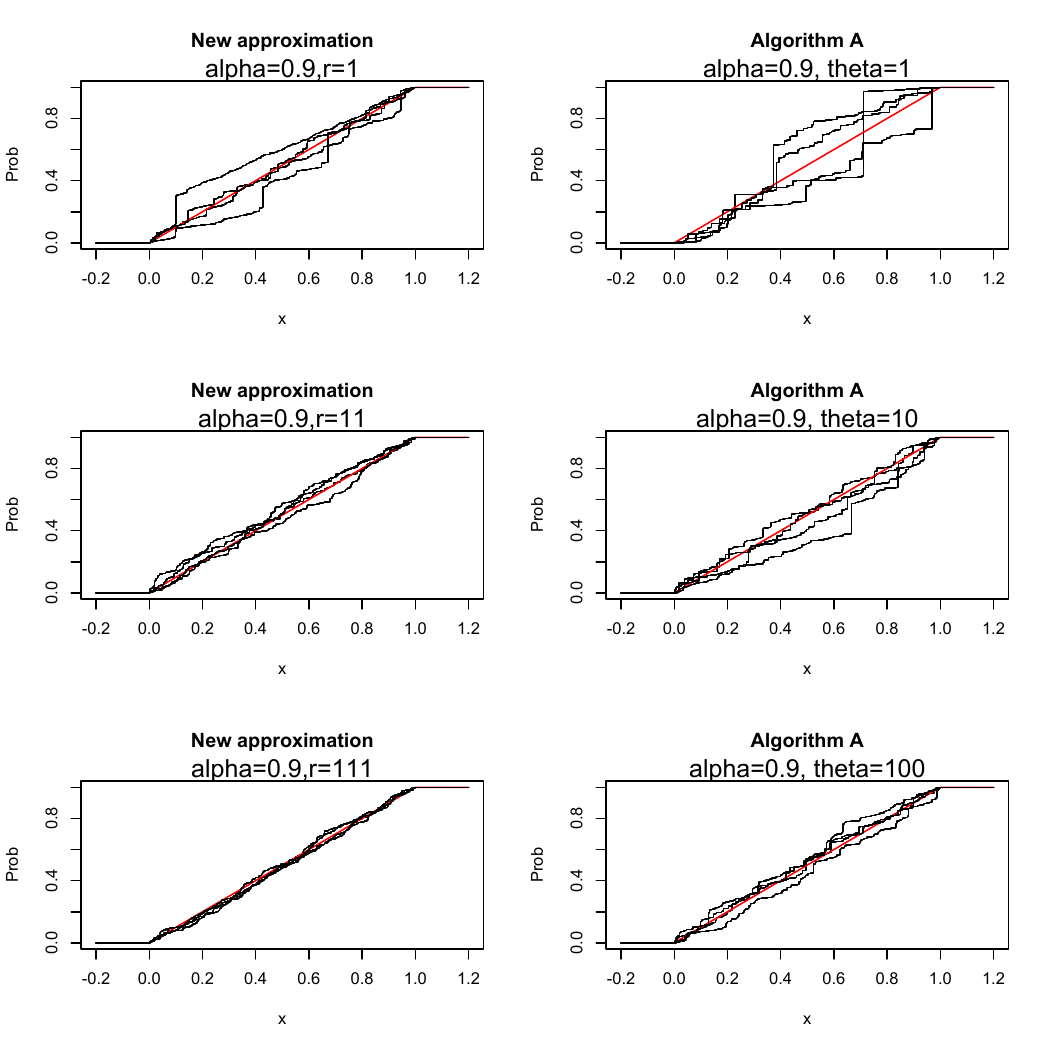}
\caption[]{Sample paths of the two-parameter Poisson-Dirichlet process, where $H$ is the uniform distribution on $[0, 1]$, $\alpha= 0.9,$ and $\theta = 1, 10, 100$. The red line denotes the cumulative distribution function of $H$.}
\label{fig3}
\end{figure} 
 
Here, we compare the approximation given in \eqref{eq:3.4cccc}  with the Algorithm A in \citet{Labadi2014} which is based on Proposition 22 of  \citet{Pitman1997}. The superiority of this approximation over the corresponding stick-breaking approximation is presented, particularly for the cases when $\alpha$ is close to 1. Since the weights
 $$\left ( \frac{L^{-1}(\Gamma_{i}/\Gamma_{r})}{ \sum_{i=r+1}^{n}L^{-1}(\Gamma_{i}/\Gamma_{r})} \right )_{r+1\leq i \leq n}$$ 
 are strictly decreasing, simulating the Poisson-Dirichlet process through representation \eqref{eq:3.4cccc} is very efficient.

Figures \ref{fig1}, \ref{fig2}, and \ref{fig3} show sample paths for the approximate Poisson-Dirichlet process with different values of $\alpha $ and $\theta~(r=\theta/\alpha)$. Clearly, the approximation \eqref{eq:3.4cccc} outperforms the approximation given in Algorithm A in \citet{Labadi2014} in all cases, as the sample paths for this approximation stay closer to the base measure $H$. This behavior agrees with Chebyshev’s inequality.
As is expected from Chebyshev’s inequality, a sample path should approach the base measure faster when either $\alpha$ or $\theta$ gets larger.

In the simulation, we set $n = 20$, $m = 20$ in Algorithm A, and $n = 20 \times 20 = 400$ in  \eqref{eq:3.4cccc}. Throughout this section, we take the base measure $H$ to be the uniform distribution on $[0,1]$. We also compute the Kolmogorov distance between the Poisson-Dirichlet process and $H$ for different values of $\alpha$ and $\theta~(r=\theta/\alpha)$. The Kolmogorov distance between $P_{n,r,L,H}$ and $H$, denoted by $d(P_{n,r,L,H},H)$, is defined by
 \begin{align}
d(P_{n,r,L,H},H)&=\sup_{x\in \mathbb{R}}|P_{n,r,L,H}(-\infty,x],H(-\infty,x]|   \nonumber \\
&:= \sup_{x\in \mathbb{R}}|P_{n,r,L,H}(x),H(x)|. \nonumber
\end{align}
For different values of $\alpha$ and $\theta$, we have obtained 500 Kolmogorov distances and reported the average of these values in Table \ref{table1}. From the simulation results in Table \ref{table1}, we can conclude that the new approach outperforms in all cases.

 \begin{table}[t!]
 \caption{This table reports the Kolmogorov distance $d(P_{n,r,L,H},H)$, where $H$ is a uniform distribution on $[0, 1]$.}
\label{table1}\par
\begin{tabular}{ccccc}
\hline
& & & New & Algorithm A \\ 
  \cline{4-5}
$\alpha$ & \multicolumn{1}{c}{$\theta$} & \multicolumn{1}{c}{$r$} & \multicolumn{1}{c}{$d$} & \multicolumn{1}{c}{$d$} \\
\hline
0.1 & 1 & 10 & 0.03133 & 0.35899  \\
0.1 & 10 & 100 & 0.03819 & 0.19301  \\
0.1 & 100 & 300 & 0.06443& 0.13835  \\
0.5 & 1 & 2 & 0.14497 & 0.24855  \\
0.5 & 10 & 20 & 0.06549 & 0.14268  \\
0.5 & 100 & 200 & 0.04606& 0.10126  \\
0.9 & 1 & 1 & 0.09294 & 0.18038  \\
0.9 & 10 & 11 & 0.05178 & 0.10100  \\
0.9 & 100 & 111 & 0.03998 & 0.07124 

\\ \hline
\end{tabular}
\end{table}

\section {Concluding Remarks}
\label{section6}

We derive the negative binomial process directly as a functional of the Poisson random measure. Then using  this derivation of the negative binomial process, we provide a generalized Poisson-Kingman distribution and also a random discrete probability measure which contains many well-known priors in nonparametric Bayesian analysis such as Dirichlet process, Poisson-Dirichlet process, normalized generalized gamma process, etc. A natural extension of the Dirichlet process as a functional of our proposed series representation for the negative binomial process is obtained. We also provide an almost sure convergent approximation for this extended Dirichlet process. 
Then the general structures of the posterior and predictive processes are given. We also justify the role of the parameter $r$ in clustering problem.
Another by-product of our proposed series representation for the negative binomial process is a new series representation for the Poisson-Dirichlet process. It is shown that an approximation based on this new representation for the Poisson-Dirichlet process is very efficient, as illustrated in a simulation study.

\bibliographystyle{ba}
\bibliography{sample}

\begin{acks}[Acknowledgments]
The research of the second author is supported by the Natural Sciences and Engineering Research Council of Canada (NSERC) with grant number RGPIN-2018-04008.
 \end{acks}

\end{document}